\RequirePackage[english]{babel}
\documentclass[12pt,twoside,reqno]{amsart}

\usepackage[OT1]{fontenc}
\usepackage{type1cm}

\usepackage{enumerate}

\usepackage{amsthm}
\RequirePackage{amsmath,amsfonts,amssymb,amsthm}

\RequirePackage[dvips]{graphicx}

\usepackage{psfrag}
\usepackage{verbatim}
\usepackage[usenames]{color}
\usepackage[dvips]{graphicx} 

  \numberwithin{equation}{section}

\usepackage[active]{srcltx}  

  \newcommand{\N}{\mathbb{N}}         
  \newcommand{\R}{\mathbb{R}}         

  \newcommand{\PP}{\mathbb{P}}        
  \newcommand{\diam}{\text{diam}}       

  \renewcommand{\AA}{\mathcal{A}}          

  \newtheorem{theorem}{Theorem}[section]
  \newtheorem{lemma}[theorem]{Lemma}
  \newtheorem{prop}[theorem]{Proposition}

  \theoremstyle{remark}
  \newtheorem{rem}[theorem]{Remark}


\begin{document}

\thanks{P.S. acknowledges support from a Leverhulme Early Career Fellowship. V.S was partly supported by the Academy of Finland project \#126976.}

\title[Sets which are not tube null]{Sets which are not tube null and intersection properties of random measures}

\author{Pablo Shmerkin}
\address{Department of Mathematics and Statistics\\
            Torcuato Di Tella University \\
            Av. Figueroa Alcorta 7350 (1428)\\
            Buenos Aires, Argentina}
\email{pshmerkin@utdt.edu}

\author{Ville Suomala}
\address{Department of Mathematical Sciences \\
         P.O. Box 3000 \\
         FI-90014 University of Oulu \\
         Finland}
\email{ville.suomala@oulu.fi}

\begin{abstract}
We show that in $\mathbb{R}^d$ there are purely unrectifiable sets of Hausdorff (and even box counting) dimension $d-1$ which are not tube null, settling a question of Carbery, Soria and Vargas, and improving a number of results by the same authors and by Carbery. Our method extends also to ``convex tube null sets'', establishing a contrast with a theorem of Alberti, Cs\"{o}rnyei and Preiss on Lipschitz-null sets. The sets we construct are random, and the proofs depend on intersection properties of certain random fractal measures with curves.
\end{abstract}

\maketitle

\section{Introduction and main results} \label{sec:introduction}

\subsection{Non-tube null sets and localisation of the Fourier transform}

By a \textit{tube} $T$ of width $w=w(T)>0$ we mean the $w$-neighborhood of some line in $\mathbb{R}^d$. We recall that a set $A\subset\mathbb{R}^d$ is called \textit{tube null} if for any $\delta>0$ it can be covered by countably many tubes $\{T_j\}$ with $\sum_j w(T_j)^{d-1}\le \delta$.

The class of tube null sets arises, perhaps surprisingly, in the localisation problem for the Fourier transform in dimension $d\ge 2$. Indeed, Carbery, Soria and Vargas \cite[Theorem 4]{CSV07} have shown (generalizing a result of Carbery and Soria in \cite{CarberySoria97}) that if $E$ is a tube null subset of the unit ball of $\R^d$ (denoted $\mathbf{B}_d$), where $d\ge 2$, then there exists $f\in L^2(\mathbb{R}^d)$ which is identically zero on $\mathbf{B}_d$, and such that the localisations
\[
S_R f(x) = \int_{|\xi|<R} \widehat{f}(\xi) \exp(2\pi i \xi\cdot x) d\xi
\]
fail to converge as $R\to\infty$ for all $x\in E$. It is an open problem whether, conversely, every set of divergence for $S_R$ is tube-null. Motivated by this connection, in \cite[p.155]{CSV07} (see also \cite{Carbery09}) the authors pose the problem of finding the infimum of the Hausdorff dimensions of sets in $\mathbb{R}^d$, which are \textit{not} tube null, and show that this infimum lies between $d-1$ and $d-1/2$. We are able to settle this question:

\begin{theorem} \label{thm:non-tube-null}
There exists a purely unrectifiable set $A\subset \mathbb{R}^d$ which has Hausdorff and box counting dimension $d-1$ and is not tube null.
\end{theorem}

We remark that the non-tube null sets of fractional dimension constructed in \cite{CSV07} are unions of spheres, and therefore fail to be purely unrectifiable.

We obtain a finer result in terms of gauge functions. Recall that a function $h:(0,\infty)\to (0,\infty)$ is called a gauge function if it is non-decreasing, continuous and $\lim_{t\downarrow 0} h(t)=0$ (sometimes continuity is not assumed). Given a gauge function $h$, the $h$-dimensional Hausdorff measure $\mathcal{H}^h$ is defined as
\[
\mathcal{H}^h(E) = \lim_{\delta\to 0} \inf\left\{ \sum_{i=1}^\infty h(\diam(E_i)) : E\subset\bigcup_i E_i, \diam(E_i)<\delta \right\}.
\]
This is always a measure on the Borel $\sigma$-algebra. When $h(t)=t^\beta$, we recover the usual $\beta$-dimensional Hausdorff measure $\mathcal{H}^\beta$. See e.g. \cite[Section 2.5]{Falconer03} for further details.

\begin{theorem} \label{thm:gauge-main-result}
Let $h:(0,\infty)\to(0,\infty)$ be a gauge function such that  $h(2t) \le 2^d h(t)$,
and
\begin{equation} \label{eq:growth-h}
\int_0^1 t^{-1}\sqrt{t^{1-d}|\log(t)| h(t)}\,dt<+\infty.
\end{equation}
Then there exists a compact set $A\subset\mathbb{R}^d$ with the following properties:
\begin{enumerate}
\item For each $n$, $A$ can be covered by $C/h(2^{-n})$ balls of radius $2^{-n}$, where $C>0$ depends only on $d$.
\item $0<\mathcal{H}^h(A)<\infty$.
\item\label{ntn} If $B\subset A$ is a Borel set with $\mathcal{H}^h(B)>0$, then $B$ is not tube null. In particular, $A$ is not tube null.
\end{enumerate}
\end{theorem}

One obtains Theorem \ref{thm:non-tube-null} by taking e.g. $h(t)=t^{d-1}|\log t\log|\log t||^{-3}$; see Section \ref{sec:proofs}.

The condition $h(2t)\le 2^d h(t)$ is very mild for a subset of $\mathbb{R}^d$. The key assumption is \eqref{eq:growth-h}; it says that $A$ is ``larger than $d-1$ dimensional by at least a logarithmic factor''. Theorem \ref{thm:gauge-main-result} fails if $\liminf_{t\downarrow 0} h(t) t^{1-d} > 0$, see \cite[Proposition 7]{CSV07}. It remains an open problem to determine the exact family of gauge functions for which non tube null sets exist.

\subsection{Tubes around more general curve families}

When $d=2$, we are also able to treat tubes around more general curves. Given a family of curves $\mathcal{F}$ in $\R^2$, we call the $w$-neighborhood of $F\in\mathcal{F}$ an \textit{$\mathcal{F}$-tube} of width $w=w(T)$. We say that a set $A\subset\mathbb{R}^2$ is \textit{$\mathcal{F}$-tube null} if, for every $\delta>0$, there is a countable covering $\{ T_j\}$ of $A$ by $\mathcal{F}$-tubes, with $\sum_j w(T_j)<\delta$.

Given $k\in\mathbb{N}$, let $\mathcal{P}_{k}$ be the family of (real) algebraic curves of degree at most $k$. Observe that $\mathcal{P}_1$-tube null is just tube null. By imposing a slightly stronger integrability condition for $h$, we obtain the following generalisation of Theorem \ref{thm:non-tube-null} for $\mathcal{P}_k$.

\begin{theorem}\label{thm:polynomial}
For $d=2$ and $k\in\N$, Theorem 1.2 continues to hold if in \eqref{ntn}, ``tube null'' is replaced by ``$\mathcal{P}_k$-tube null''. and if \eqref{eq:growth-h} is replaced by
\begin{equation} \label{eq:growth-h-polynomial}
\int_0^1 t^{-1}|\log(t)|\sqrt{t^{1-d} h(t)}\,dt<+\infty.
\end{equation}
\end{theorem}

The family of algebraic curves of a bounded degree is ``essentially finite dimensional'', see Lemma \ref{lem:almost-finite-box-dim}. These results pose the question of how large a family of curves $\mathcal{F}$ may be so that there exist sets of less than full Hausdorff dimension which are non $\mathcal{F}$-tube null. The family $\mathcal{P}=\cup_{k\in\N}\mathcal{P}_{k}$ does not have this property for trivial reasons: it is Hausdorff dense in the compact subsets of the unit square. If we instead consider the family $\mathcal{Q}\subset\mathcal{P}$ of algebraic curves which are graphs of functions of either $x$ or $y$ with derivative at most $1$, then the situation is much more subtle. Indeed, Alberti, Cs\"ornyei, and Preiss (See \cite[Theorem 2]{ACP05}) proved that for the family  $\mathcal{L}$ of $1$-Lipschitz graphs in the coordinate directions, any Lebesgue-null set is $\mathcal{L}$-tube null. By approximation, the same can be deduced to hold for $\mathcal{Q}$.

One is then led to ask what the situation is for infinite dimensional families of curves which nonetheless carry more structure than just being Lipschitz. One of the most natural such families is the following: let $\mathcal{C}$ be the family of curves which are graphs of a convex function $f:[0,1]\to \R$. It is not hard to see $\mathcal{C}$ is not doubling in the Hausdorff metric (see Section \ref{sec:proofs2}). Nevertheless, in contrast with the result of Alberti, Cs\"{o}rnyei and Preiss, there are sets of dimension less than $2$ which are not $\mathcal{C}$-tube null.

\begin{theorem} \label{thm:convex}
For every $5/3<\beta<2$, there exists a set $A\subset\R^2$ with $0<\mathcal{H}^\beta(A)<\infty$ which is not $\mathcal{C}$-tube null.
\end{theorem}
It seems very likely that the method can be pushed to show that $5/3$ can be replaced by $3/2$ in the above theorem. We do not know what is the best possible value, and conjecture that $5/3$ cannot be replaced by $1$, i.e. there is $\delta>0$ such that every set of Hausdorff dimension $1+\delta$ {\em is} $\mathcal{C}$-tube null.

\subsection{Further results}

As a corollary of the proofs, we can extend \cite[Theorem 1]{Carbery09} to one of the endpoints and a wider class of tubes.
\begin{theorem} \label{thm:bounded-density-projections}
For any $\beta>d-1$, there exists a set $A\subset\mathbb{R}^d$ with positive and finite $\beta$-dimensional Hausdorff measure $\mathcal{H}^\beta$, such that
\[
\sup_{T} \frac{\mathcal{H}^\beta(A\cap T)}{w(T)^{d-1}} <+\infty,
\]
where the supremum is over all (linear) tubes. In dimension $d=2$, the result also holds when the supremum is taken over all $\mathcal{P}_k$-tubes (for fixed $k$).
\end{theorem}

In \cite[Theorem 1]{Carbery09} this is proved, for standard tubes, with any exponent $\gamma<\min(\beta,d-1)$ in the denominator, and the question of finding all possible pairs $(\beta,\gamma)$ is posed. Examples satisfying this estimate at the other endpoint, corresponding to $\beta=\gamma<d-1$, were constructed by Orponen \cite{Orponen13} (again, only in the case of standard tubes around lines).

Our final result concerns the dimension of intersections of fractals and lines (or, more generally, algebraic curves). It is a general result of Marstrand (in the plane) and Mattila (in arbitrary dimension) that a Borel set $E\subset\mathbb{R}^d$ of Hausdorff dimension $\beta>1$ intersects ``typical'' lines in Hausdorff dimension at most $\beta-1$ (here ``typical'' refers to an appropriate natural measure space, see \cite[Theorem 10.10]{Mattila95}). It is of interest to sharpen this result for specific classes of sets. For example, Furstenberg \cite{Furstenberg70} conjectured that for certain fractals of dynamical origin, there are no exceptional lines, and Manning and Simon \cite{ManningSimon12} proved that typical lines with rational slopes intersect the Sierpi\'{n}ski carpet in dimension strictly less than $\beta-1$, where $\beta$ is the dimension of the carpet. It follows from our methods that there exist sets for which there are no exceptional lines in the Marstrand-Mattila's Theorem, in a strong uniform quantitative way.

\begin{theorem} \label{thm:fibres}
Let $h$ be a gauge function satisfying the assumptions of Theorem \ref{thm:gauge-main-result}. Then there exists a compact set $A\subset\mathbb{R}^d$ of positive and finite $h$-dimensional measure with the following property: there exists $C>0$ such that for any line $\gamma\subset\mathbb{R}^d$ and any $r>0$, the fibre $A\cap \gamma$ can be covered by $C r/h(r)$ intervals of length $r$.

If $d=2$, the same holds for all $\gamma\in\mathcal{P}_k$, provided $h$ satisfies the slightly stronger assumption of Theorem \ref{thm:polynomial}.
\end{theorem}

The proofs of all our main results rely on a random iterative construction described in the next section. Thus this paper can be seen as an application of the probabilistic method, based on the insight that it is often easier to exhibit objects with certain properties by showing that almost every object in a random family satisfies them. Although here we study the geometry of random measures as a tool towards our results, many recent articles have studied projections of random fractals for their own sake (see \cite{RamsSimon11, PeresRams11, RamsSimon14} and references therein). In particular, the results in \cite{RamsSimon11, PeresRams11} on projections of fractal percolation led us to believe that random tree-like fractals were likely not tube null, and provided several of the ideas needed to prove it.

\subsection*{Acknowledgements}

 We learned some of the ideas we use from \cite{PeresRams11}, and we thank Y. Peres and M. Rams for sharing their insights with us. We are also grateful to M. Cs\"{o}rnyei for telling us about the problems related to non tube null sets.

\section{Notation and construction} \label{sec:notation-and-construction}

We use $O(\cdot)$ notation: $X=O(Y)$ means $X\le C Y$ for some constant $0<C<+\infty$, $X=\Omega(Y)$ means $Y=O(X)$, and $X=\Theta(Y)$ means $X=O(Y)$ and $Y=O(X)$. When the implicit constants depend on some other constant, this will be denoted by subscripts; so for example $Y=O_k(X)$ means that $Y\le C(k)X$ for some positive function $C$ of $k$. Throughout the paper, we let $|A|$ denote the Lebesgue measure of a set $A\subset\R^d$. We also let $D$ denote the Hausdorff metric in the space of compact subsets of the unit cube $[0,1]^d$.

We prove all our results for sets obtained as the limit of an iterative random construction related to (although different from) fractal percolation. A somewhat related random construction was used by Peres and Solomyak \cite{PeresSolomyak05} to obtain sets $A$ such that $\mathcal{H}^h(A)>0$, yet  almost all orthogonal projections of $A$ have zero Lebesgue measure. Such sets are necessarily tube-null, and it is thus not surprising that the results of \cite{PeresSolomyak05} apply to completely different gauge functions than the ones considered in the present work.

We now describe our random construction: Let $\mathcal{D}_n$ denote the collection of closed dyadic sub-cubes of $[0,1]^d$ of side length $2^{-n}$. Let $\{a_n\}$ be a sequence satisfying
\[
a_n\in \{1,2^d\}\quad\text{ and }\quad P_n:=\prod_{i=1}^n a_i =\Theta\left(1/h(2^{-n})\right).
\]
Such sequence exists because $h(t)\le h(2t)\le 2^{d}h(t)$.

Starting with the unit cube $A_0=[0,1]^d$, we inductively construct random sets $A_n$ as follows. If $a_n=2^d$, set $A_{n+1}=A_n$. Otherwise, if $a_n=1$, choose, for each $D\in\mathcal{D}_n$ such that $D\subset A_n$, one of the $2^d$ dyadic sub-cubes of $D$ (which are in $\mathcal{D}_{n+1}$), with all choices being uniform and independent of each other and the previous steps. Let $A_{n+1}$ be the union of the chosen sub-cubes. Then $\{ A_n\}$ is a decreasing sequence of nonempty compact sets, and we set $A=\bigcap_{n=1}^\infty A_n$. See Figure \ref{fig:construction} for an example.

  \begin{figure}
    \centering
    \includegraphics[width=0.9\textwidth]{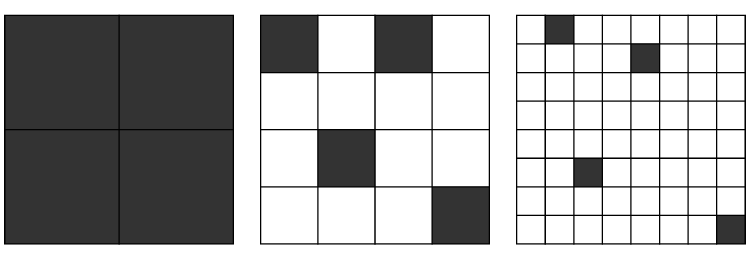}
    \caption{The first three steps in the construction of $A$ in the plane, with $a_1=4, a_2=a_3=1$} \label{fig:construction}
  \end{figure}

\section{Proof of key result} \label{sec:proof-key-lemma}

 Let $\mu_n$
be the normalized restrictions of Lebesgue measure to $A_n$, i.e.
\begin{equation}\label{mun-def}
\mu_n(B)=2^{dn} P_{n}^{-1}|B\cap A_n|
\end{equation}
 for all $B\subset\R^d$. Write $\mathcal{B}_n$ for the ring generated by the dyadic cubes in $\mathcal{D}_n$. It is easy to check that if $m\ge n$, then $\mu_m(B)=\mu_n(B)$ for $B\in\mathcal{D}_n$, so it follows from Carath\'{e}odory's extension theorem that there is a Borel probability measure $\mu$ on $\R^d$ (but supported on $A$) such that $\mu(B)=\mu_n(B)$ for any set $E\in\mathcal{B}_n$, see e.g. \cite[Proposition 1.7]{Falconer03}. In particular, $\int f d\mu_n\to \int f d\mu$ for any function $f$ which is $\mathcal{B}_j$-measurable for some $j$, and from here an approximation argument shows that $\mu_n\to\mu$ weakly. Notice that, even though $A$ is random, this convergence is deterministic for any realization of the sequence $(A_n)$.

 It is standard that $\mu(E)=\Theta_d(\mathcal{H}^h(E))$ for any Borel set $E\subset A$; we give the proof for completeness. Write
 \[
 \overline{\mathcal{H}}^h(E)=\lim_{n\to\infty} \inf\left\{ \sum_{i=1}^\infty h(2^{-k_i}) : E\subset\bigcup_i E_i, E_i\in\mathcal{D}_{k_i}, k_i\ge n \right\}.
 \]
 Since any set of diameter $r\in [2^{-n},2^{1-n})$ can be covered by $O_d(1)$ cubes in $\mathcal{D}_n$, and $\diam(Q)=\Theta_d(2^{-n})$ for $Q\in\mathcal{D}_n$, we have $\overline{\mathcal{H}}^h(E)=\Theta_d(\mathcal{H}^h(E))$ for any set $E$. Since, by construction,
 \[
 \mu(Q)=\mu_n(Q)=\Theta(h(2^{-n}))\quad\text{ if } Q\in\mathcal{D}_n\text{ and }Q\subset A_n,
  \]
 the claim follows from the fact that dyadic cubes generate the Borel $\sigma$-algebra. In particular, $0<\mathcal{H}^h(A)<\infty$ and, for any Borel set $E\subset A$, $\mathcal{H}^h(E)>0$ if and only if $\mu(E)>0$.

We now start the core of the proof of Theorem \ref{thm:gauge-main-result}: showing that almost surely, no positive measure subset of $A$ is tube null.

Let $\mathcal{A}$ denote the family of all lines which intersect the unit cube. Given $\ell\in\mathcal{A}$ and $n\in\mathbb{N}$, we define the random variable
\begin{equation}\label{eq:an}
Y_n^\ell= 2^{dn} P_n^{-1} \mathcal{H}^1(A_n\cap \ell) = \frac{\mathcal{H}^1(A_n\cap \ell)}{|A_n|},
\end{equation}
where $\mathcal{H}^1$ denotes $1$-dimensional Hausdorff measure (length). Our proof will involve estimating the $Y_n^\ell$, and indeed showing that they are uniformly bounded. This is the content of our key result:

\begin{theorem} \label{thm:main-technical-result} Almost surely,
$\sup_{n\in\mathbb{N},\ell\in\mathcal{A}} Y_n^\ell <\infty$.
\end{theorem}

Theorem \ref{thm:main-technical-result}  will follow from the next two lemmas. The first is a large deviation argument that we adapt from \cite{PeresRams11}. Since we want $Y_{n}^\ell$ to be a martingale, we only consider non-dyadic lines, e.g. lines $\ell\in\mathcal{A}$ not contained in any dyadic hyperplane, $\{(x_1,\ldots, x_d)\in\R^d\mid x_j=k 2^{-n}\}$, $k\in\mathbb{Z}$, $n\in\N$. We denote the family of non-dyadic lines by $\mathcal{A}'$. Observe that $\sup_{n\in\mathbb{N},\ell\in\mathcal{A}} Y_n^\ell= O_d\left(\sup_{n\in\mathbb{N},\ell\in\mathcal{A}'} Y_n^\ell\right)$.

\begin{lemma} \label{lem:large-deviation}
For any $\ell\in\mathcal{A}'$, $n\in \mathbb{N}$, and $\kappa>0$ for which
\begin{equation}\label{assu:kappa}
\kappa^2 2^{(1-d)n}P_n=\Omega(1)\,,
\end{equation}
we have
\[
\mathbb{P}\left(\left|Y_{n+1}^\ell-Y_n^\ell \right|> \kappa\sqrt{Y_n^\ell}\right) \le  O(1) \exp(-\Omega(1)\kappa^2 2^{(1-d)n}P_n).
\]
\end{lemma}

\begin{proof}
Fix $n$. If $a_n=2^d$ then $Y_{n+1}^\ell=Y_n^\ell$, so we assume $a_n=1$, whence $P_{n+1}=P_n$. Write $D$ for the collection of cubes in $\mathcal{D}_n$ forming $A_n$ that intersect $\ell$ in a set of positive length. In the following we condition on $D$. Let $\nu_j= 2^{n}\mathcal{H}^1|_{\ell\cap A_j}$, $j\in\{n,n+1\}$. For each $Q\in D$, we let
\[
X_Q = 2^{d}\nu_{n+1}(Q)-\nu_n(Q).
\]
Since we are conditioning on $D$, the random variables $\{ X_Q:Q\in D\}$ are independent, have zero mean,
 and are bounded in modulus by $O(1)$. For each $j$, we decompose $D$ into the families
\[
D_j = \{ Q\in D: \sqrt{d}\cdot 2^{-j}\le \nu_n(Q)< \sqrt{d} \cdot 2^{1-j} \}.
\]
Then $D_j$ is empty for all $j< 0$.  Moreover, as $Y_n^\ell = 2^{(d-1)n} P_n^{-1}\sum_{Q\in D}\nu_n(Q)$, we have
\[Y_{n}^\ell\ge \Omega(1)(\#D_j) 2^{(d-1)n-j}P_n^{-1}\,\,,\]
for all $j$.

By Hoeffding's inequality, for any $\lambda>0$ we have the estimate
\[
\PP\left(\left|\sum_{Q\in D_j} X_Q\right|>\lambda 2^{(1-d)n} P_n \sqrt{Y_n^\ell} \right) \le O(1)\exp(-\Omega(\lambda^2)2^{j+(1-d)n}P_n)\,,
\]
recall $|X_Q|=O(2^{-j})$.
Since $Y_{n+1}^\ell-Y_n^\ell=2^{(d-1)n} P_n^{-1}\sum_{Q\in D} X_Q$, we conclude that
\begin{align*}
\PP\left(\left|Y_{n+1}^\ell-Y_n^\ell\right|>\kappa\sqrt{Y_n^\ell} \right) &\le \sum_{j=1}^\infty \PP\left(\left|\sum_{Q\in D_j} X_Q\right|>\Omega(j^{-2}) 2^{(1-d)n} P_n \kappa \sqrt{Y_n^\ell}\right) \\
&\le O(1)\exp(-\Omega(1) 2^{(1-d)n} P_n \kappa^2),
\end{align*}
where we use \eqref{assu:kappa} to obtain the last estimate.
\end{proof}

The second lemma shows that a finite set of lines of exponential size controls all lines, up to an ultimately negligible error.

\begin{lemma}\label{lem:many-lines-to-all-lines}
For each $n$, there is a (deterministic) family of lines $\mathcal{A}_n\subset\mathcal{A}'$ such that $\#\mathcal{A}_n \le O(1)^n$, and
\begin{equation*}
\sup_{\ell\in \mathcal{A}'}Y^{\ell}_n\le\sup_{\ell\in \mathcal{A}_n}Y^{\ell}_n +  O(2^{-n}),
\end{equation*}
for any realization of $A$.
\end{lemma}

\begin{proof}
We construct a family of lines $\AA_n$ with $O(1)^n$ elements such that
given any line $\ell\in\AA'$, there is $\ell'\in\AA_n$ such that
\begin{equation}\label{gg'}
\mathcal{H}^1(\ell \cap Q)\le \mathcal{H}^1(\ell' \cap Q)+O(8^{-n})\quad\text{ for any } Q\in\mathcal{D}_n.
\end{equation}

We first assume that $d=2$.
Recall that $D$ is the Hausdorff distance between closed subsets of the unit cube.
By elementary geometry, there is $\AA^0_{n}\subset\AA'$ with $O(1)^n$ elements which is $(64^{-n})$-dense in the $D$ metric. That is, for every $\ell'\in\AA'$ there is $\ell\in\AA^{0}_n$ with $D(\ell,\ell')<64^{-n}$. For each horizontal dyadic line
\[
\ell_k=\{(x,y)\in\R^2\mid y=k 2^{-n}\},\quad k=0,1,\ldots,2^n,
\]
let $\AA_{\ell_k}$ denote the collection of lines forming an angle $\pm 8^{-n}$ with $\ell_k$ and crossing $\ell_k$ at any of the points $(m 8^{-n}, k 2^{-n})$, $m=0,1,\ldots,8^{n}$. Let $\AA^1_{n}$ be the union of all the $\AA_{\ell_k}$, $k=0,\ldots,2^n$. Observe that $\AA^1_{n}$ has only $O(1)^n$ elements. Finally, let $\AA^2_{n}$ be a similar family of lines constructed around the vertical dyadic lines $\{(x,y)\in\R^2\mid x=k 2^{-n}\}$ and define $\AA_n=\AA^0_{n}\cup\AA^1_{n}\cup\AA^2_{n}$.

Now let $\ell\in\AA'$.
If the angle between $\ell$ and the coordinate directions is larger than $8^{-n}$ or if $\ell$ is completely contained in a single row or column in $\mathcal{D}_n$,
then by elementary geometry, \eqref{gg'} holds for any $\ell'\in\AA^0_n$ which is $64^{-n}$ close to $\ell$ in the $D$ metric. Otherwise, we can find a line $\ell'$ from $\AA^1_{n}$ or $\AA^2_{n}$ such that \eqref{gg'} is fulfilled.

When $d>2$, we also let $\AA^0_{n}$ be a $(64^{-n})$-dense family in $\AA'$ with $O(1)^n$ elements. As in the $d=2$ case, for any $\ell\in\AA'$ which forms an angle at least $8^{-n}$ with all the coordinate hyperplanes, there is $\ell'\in\AA^1_{n}$ so that \eqref{gg'} is satisfied. To deal with the lines forming a small angle with at least one hyperplane
\[H=\{(x_1,\ldots,x_d)\in\R^d\mid x_j=k 2^{-n}\}\,,\]
which they intersect, we need to construct families $\AA^1_{n},\ldots, \AA^d_{n}$. This is done in a similar way to the $d=2$ case, by considering a dense enough subset $Y_H\subset H$ with $O(1)^n$ elements and choosing, for each $z\in Y_H$, suitable lines that cross $H$ at $z$ and are almost parallel to $H$. We omit the details.

We can now finish the proof of the lemma. Since each line $\ell$ hits at most $O(2^n)$ squares in $\mathcal{D}_n$, we conclude from \eqref{gg'} that $\mathcal{H}^1(\ell\cap A_n)\le\mathcal{H}^1(\ell'\cap A_n)+O(4^{-n})$ and combined with the definition of $Y_n^\ell$, this implies that
$Y_n^\ell\le Y_n^{\ell'}+2^{dn} P_n^{-1}O(4^{-n})=Y_n^{\ell'}+O(2^{-n})$. Recall that $P_n=\Omega(2^{(d-1)n})$. It follows that $\AA_n$ is the desired family.
\end{proof}

\begin{proof}[Proof of Theorem \ref{thm:main-technical-result}]
Let $M_n = \sup_{\ell\in\AA'} Y_n^\ell$. It follows from \eqref{eq:growth-h} that
\begin{equation} \label{eq:bound-sum}
\sum_{n=1}^\infty\sqrt{n 2^{(d-1)n} h(2^{-n})}<\infty.
\end{equation}
We claim that it is enough to find $C<\infty$ such that
\begin{equation} \label{eq:borel-cantelli}
\sum_{n=1}^\infty \PP\left(M_{n+1}-M_n>C \sqrt{n 2^{(d-1)n} h(2^{-n}) M_n}+ O(2^{-n})\right)
<\infty.
\end{equation}
Indeed, if this is true then, by the Borel-Cantelli lemma there is $n_0$ such that
\[
M_{n+1} \le M_n+ C \sqrt{n 2^{(d-1)n} h(2^{-n}) M_n}+ O(2^{-n})
\]
for all $n\ge n_0$. Hence, $M_n\le \overline{M}_n$ for all $n\ge n_0$, where $\overline{M}_{n_0}=M_{n_0}$ and
\[
\overline{M}_{n+1} = \overline{M}_n+ C \sqrt{n 2^{(d-1)n} h(2^{-n}) \overline{M}_n}+ O(2^{-n}).
\]
Dividing through by $\sqrt{\overline{M}_n}$ and using that $\overline{M}_n$ is increasing, we get
\[
\sqrt{\overline{M}_{n+1}} \le \sqrt{\overline{M}_n} +  C \sqrt{n 2^{(d-1)n} h(2^{-n})} + O(2^{-n})/\sqrt{M_{n_0}}.
\]
In light of \eqref{eq:bound-sum}, $\sqrt{\overline{M}_n}$ is uniformly bounded, and hence so is $M_n$, giving the claim.

Hence the task is to verify \eqref{eq:borel-cantelli}, and it is enough to do so if we fix $n$ and condition on $A_n$ (so long as the constant $C$ is independent of $A_n$). Pick $\ell\in\AA$. Recalling that $P_n^{-1}=\Theta(h(2^{-n}))$, it follows from
Lemma \ref{lem:large-deviation} that
\[
\PP\left(Y_{n+1}^\ell - Y_n^\ell > C \sqrt{n 2^{(d-1)n} h(2^{-n}) Y_n^\ell}\right)  \le O(1) \exp(-C^2\,\Omega(n)).
\]
Observe that $n 2^{(d-1)n} h(2^{-n}) 2^{(1-d)n} P_n=\Omega(n)=\Omega(1)$ so that \eqref{assu:kappa} holds and we may apply Lemma \ref{lem:large-deviation}.

Let $\AA_{n+1}$ be the family given by Lemma \ref{lem:many-lines-to-all-lines}. For $C$ sufficiently large, it holds that
\begin{align}\label{eq:poly-times-exp}
\PP\left(\max_{\ell\in \AA_{n+1}} Y_{n+1}^\ell - M_n \ge C \sqrt{n 2^{(d-1)n} h(2^{-n}) M_n} \right) &\le O(1)^n \exp(-C^2\Omega(n))\\
&=O(\exp(-\Omega(n))).\notag
\end{align}
In light of Lemma \ref{lem:many-lines-to-all-lines}, we see that \eqref{eq:borel-cantelli} holds, completing the proof.
\end{proof}

\section{Proofs of Theorems \ref{thm:non-tube-null} and \ref{thm:gauge-main-result} } \label{sec:proofs}

\begin{proof}[Proof of Theorem \ref{thm:gauge-main-result}]
The first two assertions in Theorem \ref{thm:gauge-main-result} are clear; we only need to show that if $B\subset A$ has positive $\mu$-measure, then $B$ is not tube null. By Theorem \ref{thm:main-technical-result}, almost surely there is $C>0$ such that $Y_n^\ell\le C$ for all $n$ and $\ell$. Let $\pi_H$ denote the orthogonal projection onto a hyperplane $H\subset\R^d$ and $\mu_n^H(B)=\mu_n(\pi_H^{-1}(B))$ for all $B\subset H$. Using Fubini's theorem, $\mu_n^H$ has density 
\[\lim_{r\downarrow 0}\frac{\mu_n^H(B(x,r))}{|B(x,r)|}\le Y_n^{\ell(H,x)}\,,\]
where $\ell(H,x)$ is the line passing through $x\in H$ orthogonal to $H$ and $|B(x,r)|$ denotes the $(d-1)$-dimensional Lebesgue measure of the ball $B(x,r)\subset H$. This implies that for each $n$, all orthogonal projections of $\mu_n$ onto hyperplanes have a density (w.r.t $(d-1)$-dimensional Lebesgue measure) uniformly bounded by $C$. The same therefore holds for $\mu$.

Now this implies that if $\{ T_j\}$ is a countable collection of tubes covering $B$, then
\[
0<\mu(B) \le \sum_j \mu(T_j) \le \sum_j C\,w(T_j)^{d-1},
\]
showing that $B$ is not tube null.
\end{proof}

\begin{proof}[Proof of Theorem \ref{thm:non-tube-null}]
Take $h(t)=t^{d-1}|\log t\log|\log t||^{-3}$. Because $\lim_{t\downarrow 0}\log h(t)/\log t=d-1$, $A$ has Hausdorff and box dimension equal to $d-1$ and is a.s. not tube null by Theorem \ref{thm:gauge-main-result}.

It remains to show that $A$ is purely unrectifiable. For simplicity, we assume that $d=2$. Denote by $\mathcal{N}$ the collection of all $n\in\N$ for which $a_n=a_{n+1}=1$. Observe that $a_n$ in \eqref{eq:an} can be selected so that $\mathcal{N}$ is infinite.
Suppose on the contrary, that there is a continuously differentiable curve $\Gamma$ such that $\Gamma\cap A$ has positive length. By the Lebesgue density theorem, $\mathcal{H}^1$-almost all $x\in A\cap\Gamma$ satisfy
\begin{align}\label{eq:densitypoint}
\lim_{n\rightarrow\infty}\frac{\mathcal{H}^1\left(3 Q_n\cap \Gamma\cap A\right)}{\mathcal{H}^1\left(3 Q_n\cap \Gamma\right)}=1,\quad\
\limsup_{n\rightarrow\infty}\frac{\mathcal{H}^1\left(3 Q_n\cap \Gamma\right)}{2^{-n}}= O(1),
\end{align}
where $Q_n\in\mathcal{D}_n$ is a square that contains $x$ and $3Q_n$ is the union of $Q_n$ and its neighbors in $\mathcal{D}_n$. Fix $x\in A\cap\Gamma$ satisfying \eqref{eq:densitypoint} and let $n\in\mathcal{N}$. Since each of the neighboring squares of $D_n$ contain only at most one square from $\mathcal{D}_{n+2}$, $\Gamma\cap 3Q_n$ has to cross at least one column or row $S$ of squares in $\mathcal{D}_{n+2}$ such that $3Q_n\cap A\cap S=\emptyset$. This implies that $\mathcal{H}^1(3Q_n\cap\Gamma\setminus A)\ge\Omega(2^{-n})$. For large $n$ this yields a contradiction with \eqref{eq:densitypoint}.

The case $d>2$ follows with the same argument, assuming that $a_n=a_{n+1}=\ldots=a_{n+O_d(1)}=1$ for infinitely many $n$.
\end{proof}

\section{Proof of Theorems \ref{thm:polynomial}, \ref{thm:bounded-density-projections} and \ref{thm:fibres}}

The proof of Theorem \ref{thm:polynomial} follows the same pattern of the proof of Theorem \ref{thm:gauge-main-result}; the main difference lies in establishing the analog of Lemma \ref{lem:many-lines-to-all-lines}, which requires a more involved argument. From now on, we assume that $d=2$. Fix $k\in\mathbb{N}$. As in the linear case, for any curve $\gamma$ and $n\in\mathbb{N}$, we define the random variable
\begin{equation}\label{Yn-def}
Y_n^\gamma= 4^n P_n^{-1} \mathcal{H}^1(A_n\cap \gamma).
\end{equation}
Lemma \ref{lem:large-deviation} continues to hold for $\gamma\in\mathcal{P}_k$ with the same proof, unless $\gamma$ is a dyadic line of the form $\{x=k 2^{-n}\}$ or $\{y=k 2^{-n}\}$, $k,n\in\N$. Indeed, in Lemma \ref{lem:large-deviation}, the only time we used the fact that we were dealing with lines was in the estimate $\mathcal{H}^1(\ell\cap Q)=O(\diam(Q))$ for all squares $Q$, and it is clear that $\mathcal{H}^1(\gamma\cap Q)=O_k(\diam(Q))$ if $\gamma\in\mathcal{P}_k$.

Before discussing the needed analog of Lemma \ref{lem:many-lines-to-all-lines}, we make a reduction that will be useful also later. The following lemma should be well known, but we include a proof for completeness.
\begin{lemma} \label{lem:bezout}
Each $\gamma\in\mathcal{P}_k$ can covered by $O_k(1)$ curves, each of which is, after a rotation by $\pi/2$ and/or a reflection, the graph of a convex, increasing function  $f\colon [a,b]\rightarrow[0,1]$ with derivative bounded by $1$.
\end{lemma}
\begin{proof}
We may assume that $\gamma=P^{-1}(0)$ where $P$ is irreducible (otherwise, apply the argument to its $O_k(1)$ irreducible factors). We also assume, as we may, that $\gamma$ does not contain (and therefore is not) a line. In particular, this implies that the partial derivatives $P_x$, $P_y$ are not identically zero. By Bezout's Theorem, the set
\[
S=P^{-1}(0)\cap \big(P_x^{-1}(0)\cup P_y^{-1}(0)\big)
\]
has cardinality $O_k(1)$. It is well known that $\gamma$ has $O_k(1)$ connected components, see e.g. \cite[Theorem 4.6]{Coste02}. It follows that $\gamma\setminus S$ can be partitioned into $O_k(1)$ curves which are graphs of functions of either $x$ or $y$, without critical points. If $(x,y)$ is in the graph of one such function, say $y=f(x)$, then implicit differentiation gives $f'(x)=-P_x(x,y)/P_y(x,y)$, and one more implicit differentiation yields
\[
f''(x) = -\frac{(f'(x))^2 P_{y y}(x,y)+ 2 f'(x) P_{xy}(x,y)+ P_{xx}(x,y)}{P_y(x,y)}.
\]
Hence, if $S'=\{ x: f'(x)=1\}$ and $S''=\{ x: f''(x)=0\}$, then the union $S'\cup S''$ has cardinality $O_k(1)$ by another application of Bezout's Theorem (since $\gamma$ is not a line). The closures of the connected components of $\gamma\setminus (S\cup S'\cup S'')$ are the required curves; their union covers all of $\gamma$ except the isolated points (if any). Note that the isolated points lie in $S$ and can thus be covered by $O_k(1)$ curves of the required type.
\end{proof}

The core of the proof of Theorem \ref{thm:polynomial} is again to show that
\begin{equation*}
C:=\sup_{\gamma\in\mathcal{P}_{k},\,n\in\mathbb{N}} Y_n^\gamma < +\infty.
\end{equation*}
By Lemma \ref{lem:bezout}, it is enough to show this for the family $\mathcal{Q}_k$, which consists of those subsets of the algebraic curves in $\mathcal{P}_k$, which are graphs of convex increasing functions with right derivative bounded from above by $1$.

The following simple Lemma is essential in
the proof of Lemma \ref{lem:many-algebraic-to-all-algebraic}.
It should be well known, but we have not been able to find a reference so a proof is included for completeness.

\begin{lemma}\label{lem:convex}
Let $f_1, f_2$ be convex increasing functions defined on $[0,1]$ with right derivative bounded above by $1$, and let $\gamma_i$, $i=1,2$, be their graphs. Then
\[
|\mathcal{H}^1(\gamma_1)-\mathcal{H}^1(\gamma_2)|=O(|f_1-f_2|_\infty)\,,
\]
where $|h|_\infty=\sup_{x\in[0,1]}|h(x)|$.
\end{lemma}

\begin{proof}
By approximation, we can assume that $f_1, f_2$ are twice continuously differentiable. Then
\begin{align}\label{fg}
\mathcal{H}^1(\gamma_1)-\mathcal{H}^1(\gamma_2)&=\int_{t=0}^1\sqrt{1+f_1'(t)^2}-\sqrt{1+f_2'(t)^2}\,dt\nonumber\\
&=\int_{t=0}^1 a(t)\left(f_1'(t)-f_2'(t)\right)\,dt\,,
\end{align}
where
\[
a(t)=\left(f_1'(t)+f_2'(t)\right)\left(\sqrt{1+f_1'(t)^2}+\sqrt{1+f_2'(t)^2}\right)^{-1}\,.
\]
Then
\[
a'(t) = b_1(t) f''_1(t) + b_2(t) f''_2(t),
\]
where $b_1,b_2$ are continuous functions on $[0,1]$ bounded by $O(1)$. Integrating by parts, we deduce from \eqref{fg} that
\begin{align*}
|\mathcal{H}^1(\gamma_1)-\mathcal{H}^1(\gamma_2)| &\le O(|f_1-f_2|_\infty) + \int_{t=0}^1 |a'(t)| |f_1(t)-f_2(t)| dt\\
&\le  O(|f_1-f_2|_\infty)(1+ |a'|_1),
\end{align*}
where $|\cdot|_1$ denotes the $L^1$ norm on $[0,1]$. But
\[
|a'|_1 \le |b_1|_\infty |f_1''|_1 + |b_2|_\infty |f_2''|_1 = O(1),
\]
using that $|f''_i|_1 = f_i'(1)-f_i'(0)\le 1$ for $i=1,2$, thanks to our assumptions.
\end{proof}

We shall next provide a simple geometric argument implying a bound on the number of $\delta$-balls needed to cover $\mathcal{Q}_k$. Recall that $D$ is the Hausdorff distance on the unit cube.

\begin{lemma}\label{lem:almost-finite-box-dim}
For all $0<\delta<1$, $\mathcal{Q}_k$ can be covered by $\exp(O_k(|\log\delta|^2))$ balls of radius $\delta$ in the $D$-metric.
\end{lemma}

\begin{proof}
We prove that given $\gamma\in\mathcal{Q}_k$ and $0<r<1$, we may cover the ball $B(\gamma,r)$ by $O_k(r^{-O_k(1)})$ balls of radius $r/2$.
It then follows by induction on $n$ that $\mathcal{Q}_k=B(\gamma_0,O(1))$ can be covered by $2^{O_k(n^2)}$ balls of radius $2^{-n}$. Given $\delta\in (0,1]$, applying this to $n$ such that $2^{-n}\le \delta< 2^{1-n}$ yields the claim.

Fix $\gamma\in\mathcal{Q}_k$, and for $\widetilde{\gamma}\in B(\gamma,r)$, let $\widetilde{f}:[a_{\widetilde{\gamma}},c_{\widetilde{\gamma}}]\to [0,1]$ be the convex increasing function with graph $\widetilde{\gamma}$. Also, let $\gamma$ be the graph of the corresponding function $f:[a_\gamma,c_\gamma]\to [0,1]$.

For notational convenience, we assume that $\tfrac{5}{r}\in\N$. For $-8\le i\le 8$, let $f_i=f+\tfrac{ir}{5}$. We extend the functions $f_i$ to $[a_\gamma-r,c_\gamma+r]$ by setting $f_i(t)=f_i(a_\gamma)$ for $a_\gamma-r\le t<a_\gamma$ and $f_i(t)=f_i(c_\gamma)$ for $c_\gamma< t\le c_\gamma+r$.
To each $\widetilde{\gamma}\in B(\gamma,r)$ we attach a sequence
\[p=p(\widetilde{\gamma})=(p_{0},p_1,\ldots,p_{5/r})\in\{-\infty,-8,-7,\ldots,7,+\infty\}^{5/r}\]
such that.
\begin{equation*}
p_j=
\begin{cases}
-\infty\text{ if }a_{\widetilde{\gamma}}>\frac{jr}{5}\,\\
i\text{ if }f_i\left(\frac{jr}{5}\right)\le\widetilde{f}\left(\frac{jr}{5}\right)<f_{i+1}\left(\frac{jr}{5}\right)\,\\
+\infty\text{ if }c_{\widetilde{\gamma}}<\frac{jr}{5}\,.
\end{cases}
\end{equation*}

By Bezout's theorem, for any $i$ and $\widetilde{\gamma}\in B(\gamma,r)$, $\widetilde{\gamma}$ intersects the graph of $f_i$ at most $O_k(1)$ times (or otherwise $\widetilde{f}= f_i$). This means that for each $\widetilde{\gamma}\in B(\gamma,r)$, there are at most $O_k(1)$ values $p_j$ such that $p_{j+1}\neq p_j$. Thus, the number of all possible sequences $p(\widetilde{\gamma})$ for $\widetilde{\gamma}\in B(\gamma,r)$ is at most $O( r^{-O(1)})$. In addition, if $p(\widetilde{\gamma})=p(\hat{\gamma})$, it follows from the construction that $\hat{\gamma}\in B(\widetilde{\gamma},r/2)$, recall that the derivative of each $\widetilde{f}\in\mathcal{Q}_k$ is between $0$ and $1$. Combining these observations implies that $B(\gamma,r)$ may be covered by $O(r^{-O(1)})$ balls of radius $r/2$.
\end{proof}

\begin{rem}
It seems likely that the bound $\exp(O_k(|\log\delta|^2))$ in Lemma \ref{lem:almost-finite-box-dim} could be improved to $\delta^{-O_k(1)}$ (this is equivalent to $\mathcal{Q}_k$ having finite box-dimension in the $D$-metric). If this is the case, then \eqref{eq:growth-h-polynomial} in Theorem \ref{thm:polynomial} can be replaced by \eqref{eq:growth-h}. However, we have not been able to prove this nor could we track such a result in the literature. Recall that there is only a mild difference between the conditions \eqref{eq:growth-h-polynomial} and \eqref{eq:growth-h}, and also that it is not known if \eqref{eq:growth-h} is sharp for Theorem \ref{thm:gauge-main-result}.
\end{rem}

As earlier in the case of lines, we ignore the elements of $\mathcal{Q}_k$ that contain a nontrivial line segment of some dyadic line $\{y=k 2^{-n}\}$, $k,n\in\N$. We denote the corresponding family by $\mathcal{Q}'_k$. Note that trivially, Lemma \ref{lem:almost-finite-box-dim} applies also for $\mathcal{Q}'_k$.

\begin{lemma} \label{lem:many-algebraic-to-all-algebraic}
For each $n$, there is a  family of curves $\mathcal{Q}_{n,k}\subset \mathcal{Q}'_k$ such that $\#\mathcal{Q}_{n,k}\le\exp(O_k(n^2))$, and
\begin{equation*}
\sup_{\gamma\in \mathcal{Q}'_k}Y^{\gamma}_n\le\sup_{\gamma\in \mathcal{Q}_{n,k}}Y^{\gamma}_n +  O_k((4/5)^n),
\end{equation*}
for any realization of $A$.
\end{lemma}
\begin{proof}

To begin with, take $\delta=25^{-n}$ and let $\Gamma'\subset\mathcal{Q}'_k$ be the $\delta$-dense family of size $\exp(O(n^2))$ given by Lemma \ref{lem:almost-finite-box-dim}. We will next modify $\Gamma'$ by adding a finite number of translates of each $\gamma\in\Gamma'$: Let $\gamma\in\Gamma'$, and let $\gamma$ be the graph of $f\colon[a,c]\rightarrow[0,1]$. Let $(a,b)$ be the interval on which $f'(x^+)<5^{-n}$. If there is $k\in\N$ such that $|f(a)-k 2^{-n}|\le 5^{-n}$, we choose numbers $-O(5^{-n})<y_i<O(5^{-n})$ for each $a\le i 5^{-n}\le b$, $i\in\N$, such that the function $f_i(x)=f(x)+y_i$ crosses the dyadic line $y=k 2^{-n}$ at $x_i=i 5^{-n}$. Let $\mathcal{E}_\gamma$ consist of $\gamma$ and all the graphs of $f_i$.
We define
\[
\Gamma=\bigcup_{\gamma\in\Gamma'}\mathcal{E}_\gamma\,.
\]

Since each $\mathcal{E}_\gamma$ contains at most $O(5^n)$ elements, it follows that the cardinality of $\Gamma$ is $\exp(O(n^2))$.
Moreover, using Lemma \ref{lem:convex} it can be checked that, for any $\gamma\in\mathcal{Q}_k$, there is $\widetilde{\gamma}\in\Gamma$ such that
\[
\mathcal{H}^1(\gamma\cap Q)\le \mathcal{H}^1(\widetilde{\gamma}\cap Q) + O_k(5^{-n})\quad\text{for all }Q\in\mathcal{D}_n.
\]
We leave the verification of the several simple cases to the reader, or see Lemma \ref{lem:many-convex-to-all-convex} for a similar but more complicated argument.

As in the proof of Lemma \ref{lem:many-lines-to-all-lines}, the claim now follows by adding over all chosen $Q$, using the trivial bound $P_n=\Omega(2^n)$, and recalling the definition of $Y_n^\gamma$ (see \eqref{Yn-def}).
\end{proof}

\begin{proof}[Proof of Theorem \ref{thm:polynomial}]
 Once we have analogs of Lemmas \ref{lem:large-deviation} and \ref{lem:many-lines-to-all-lines}, the proof of Theorem \ref{thm:main-technical-result} works verbatim to yield that almost surely
\begin{equation*}
C:=\sup_{\gamma\in\mathcal{Q}_{k},\,n\in\mathbb{N}} Y_n^\gamma < +\infty.
\end{equation*}
(Replacing Lemma \ref{lem:many-lines-to-all-lines} by Lemma \ref{lem:many-algebraic-to-all-algebraic} and \eqref{eq:growth-h} by \eqref{eq:growth-h-polynomial}, the upper bound in \eqref{eq:poly-times-exp} reads $\exp\left(O(n^2)-C^2\Omega(n^2)\right)$.)

To conclude the proof, fix some $\gamma\in\mathcal{Q}_k$; suppose $\gamma$ is the graph of $f:[a,b]\to \mathbb{R}$. Then for any $\delta>0$,
\[
\gamma(\delta) \subset B((a,f(a)),\delta) \cup B((b,f(b)),\delta) \cup \{ (x,y): |y-f(x)|<2\delta \}\,.
\]
Comparing the definitions of $\mu_n$ and $Y_\gamma^{n}$ (see \eqref{mun-def} and \eqref{Yn-def}), it then follows from Fubini's theorem that $\mu_n(\gamma(\delta)) \le O(C\delta)$, and hence the same bound holds for $\mu$ and all $\gamma\in\mathcal{P}_k$. The proof then finishes as in the case of lines.
\end{proof}

\begin{proof}[Proof of Theorem \ref{thm:bounded-density-projections}]
Take $h(t)=t^\beta$ with $\beta>d-1$. Then the random measure $\mu$ satisfies $\mu(T)=\Theta(\mathcal{H}^\beta(T\cap A))$ for any Borel set $T$. In the proof of Theorem \ref{thm:gauge-main-result} we observed that (as an easy consequence of Theorem \ref{thm:main-technical-result}) $\sup_T \mu(T)/w(T)^{d-1}<\infty$ where the supremum is over all tubes in $\R^d$. Likewise, in the proof of Theorem \ref{thm:polynomial}, the same was proved for tubes around algebraic curves in $\R^2$. The theorem follows.
\end{proof}

\begin{proof}[Proof of Theorem \ref{thm:fibres}]
We consider first the case of lines in $\mathbb{R}^d$. By Theorem \ref{thm:main-technical-result} and Fubini, there exists $C>0$ and a realization of the random set $A$ such that $\mu_n(\gamma(\delta)) \le C\delta$ for all $\delta>0,n\in\mathbb{N}$ and all lines $\gamma$. In particular, this holds for $\delta:=\sqrt{d}\cdot 2^{-n}$. Since any chosen cube in $A_n$ which intersects $\gamma$ is then contained in $\gamma(\delta)$, it follows from the definition of $\mu_n$ that $\gamma$ can intersect at most $O(2^{-n}/h(2^{-n}))$ such cubes. From here the theorem follows easily.

The situation for algebraic curves is identical, using the proof of Theorem \ref{thm:polynomial} instead.
\end{proof}

\section{Proof of Theorem \ref{thm:convex} \label{sec:proofs2}}

\subsection{Initial reductions}

The proof of Theorem \ref{thm:convex} follows once again a similar pattern. However, bounding the number of $D$-balls of radius $\delta$ needed to cover $\mathcal{C}$ is more delicate (and the bound is much larger than for the case of $\mathcal{P}_k$).

We start with some notation and reductions.
Abusing notation slightly, we will sometimes identify functions $f:[0,1]\to[0,1]$ with their graphs.
We denote by $\mathcal{C}^+$ the subset of $\mathcal{C}$ consisting of non-decreasing functions with right derivative bounded above by $1$. We note that since every curve in $\mathcal{C}$ is the union of at most four curves which are obtained from a curve in $\mathcal{C}^+$ by a possible $\pi/2$ rotation and/or a reflection, it is enough to prove Theorem \ref{thm:convex} for $\mathcal{C}^+$ instead of $\mathcal{C}$. (To be more precise, an arbitrary $f\in\mathcal{C}$ is the union of such four curves defined on some interval $[a,b]$ rather than $[0,1]$; by continuing them linearly to the left of $a$ and the right of $b$, there is no harm in assuming they are defined on all of $[0,1]$.)

\subsection{Bounding the size of $\mathcal{C}^+$}

\begin{prop} \label{prop:dim-convex-functions}
For every $0<\delta<1$, $\mathcal{C}^+$ contains a $\delta$-dense subset (in the Hausdorff metric $D$) with  $\exp(O(\delta^{-1/2}|\log\delta|))$
elements.
\end{prop}

The idea of the proof of Proposition \ref{prop:dim-convex-functions} is to associate to each $f\in\mathcal{C}^+$ a finite collection of numbers, in such a way that knowing each of these numbers with an error up to $\delta$ allows to construct a piecewise affine approximation which is within distance $O(\delta)$ of $f$. The problem is then reduced to a counting problem in a much more straightforward space. Of course, this can be done with any continuous function; the trick is to exploit the convexity and monotonicity of $f$ to show that, in essence, $\exp(O(\delta^{-1/2}|\log\delta|))$ numbers suffice to reconstruct $f$ up to error $O(\delta)$.

From now on, we assume that $\delta^{-1/2}$ is an integer $N$ (for simplicity of notation). Let us first define the parameter space. Let $X=\{0,\frac{1}{N^{2}},\ldots,\frac{N^{2}-1}{N^{2}},1\}$ and let
$\Lambda$ be the family of all increasing functions $f\colon Y\rightarrow X$, where $Y\subset X$ has at most $2N+1$ elements. It is straightforward to check that $\#\Lambda
\le N^{O(N)}$.

We reduce the proof of Proposition \ref{prop:dim-convex-functions} to the following:
\begin{prop}\label{prop:P}
There is a mapping $P\colon\mathcal{C}^+\rightarrow\Lambda$ such that if $f,\widetilde{f}\in\mathcal{C}^+$ and $P(f)=P(\widetilde{f})$, then $D(f,\widetilde{f})=O(N^{-2})$\,.
\end{prop}
This indeed implies Proposition \ref{prop:dim-convex-functions}: the needed $O(\delta)$-dense collection in $\mathcal{C}^+$ is obtained by choosing one element from $P^{-1}(\lambda)$ for each $\lambda\in P(\mathcal{C}^+)$.

To define the projection $P$, we fix $f\in\mathcal{C}^+$. We first construct $Y=Y_f$ inductively as follows: Let $x_0=0$. If $x_k<1$ is defined, let
\[x_{k+1}=\min\left\{1,x_k+\frac1N,\inf\{x\in X\,:\,x_k<x\,, f'(x^+)\ge f'(x_k^+)+\tfrac1N\}\right\}\,.\]
We stop the construction when $x_k=1$, and set $Y=\{x_0,x_1,x_2,\ldots, x_k\}$.

\begin{lemma}\label{lem:Y}
For each $f\in\mathcal{C}^+$, the set $Y_f$ satisfies
\begin{gather}
\label{Nbound}\#Y\le 2N+1,\\
\label{1N} \frac{1}{N^2}<x_{i}-x_{i-1}\le\frac{1}{N}\text{ for each }0<x_i\in Y.
\end{gather}
In addition, for each $0<x_{i}\in Y$, we have
\begin{gather}
\label{ab}|f'(t)-\frac{f(\widetilde{x}_i)-f(x_{i-1})}{\widetilde{x}_{i}-x_{i-1}}|=O(1/N)\text{ for }x_{i-1}\le t\le \widetilde{x}_i\,,
\end{gather}
where $\widetilde{x}_i=x_i-N^{-2}$.
\end{lemma}

\begin{proof}
The first claim follows directly from the construction of $Y$ since $0\le f'(t^+)\le 1$ is non-decreasing by  convexity. The claim \eqref{1N} is also clear. The last claim follows from \eqref{1N} and the fact $f'(x^+_{i-1})\le f'(t^+)\le f'(\widetilde{x}_{i}^+)\le f'(x^+_{i-1})+N^{-1}$ for all $x_{i-1}<t<\widetilde{x}_i$.
\end{proof}

We may now complete the definition of $p=P(f)$. For each $x\in Y$, we let $p(x)=\frac{k-1}{N^{2}}$, where $k\in\N$ and  $\frac{k-1}{N^{2}}\le f(x)<\frac{k}{N^{2}}$. Given $f\in\mathcal{C}^+$, we extend $p=P(f)$ to $[0,1]$ by interpolating it linearly between the points of $Y_f$. For notational convenience, we denote the extension also by $p$.  See Figure \ref{fig:convex-approximation} for an illustration.

\begin{figure}
    \centering
    \includegraphics[width=0.9\textwidth]{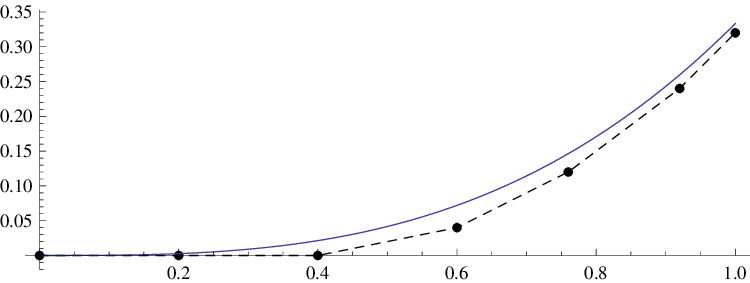}
    \caption{The graphs of $f$ and $p=P(f)$ for $f(x)=x^3/3$, $N=5$. In this case, $Y_f=\{0,\frac{5}{25},\frac{10}{25},\frac{15}{25},\frac{19}{25},\frac{23}{25},1\}$.} \label{fig:convex-approximation}
\end{figure}

  The claims of the following lemma are simple consequences of the definitions.
\begin{lemma} For $f\in\mathcal{C}^+$, $p=P(f)$, and $Y=Y_f$ it holds
\begin{gather}\label{hdef}
|p(x)-f(x)|=O(N^{-2})\text{ for all }x\in Y\,,\\
\label{h'}
p'(t)=O(1)\text{ for all }0< t< 1\,,t\notin Y\,.
\end{gather}
\end{lemma}

\begin{proof}[Proof of Proposition \ref{prop:P}]
We claim that for each $f\in\mathcal{C}^+$, the (extended) projection $p=P(f)$ satisfies
\begin{equation}\label{eq:finalclaim}
|p-f|_\infty=O(N^{-2})\,.
\end{equation}
This implies the claim since if $f,\widetilde{f}\in\mathcal{C}^+$ and $p=P(f)=P(\widetilde{f})$, then
\[
D(f,\widetilde{f})=O(|f-\widetilde{f}|_\infty)=O\left(|f-p|_\infty+|\widetilde{f}-p|_\infty\right)=O(N^{-2})\,.
\]
In short, the estimate \eqref{eq:finalclaim} holds because $Y=Y_f$ has been constructed so that the variance of $f'$ on each interval $[x_{i-1},x_i-O(N^{-2})]$ is at most $O(N^{-1})$ and $p|_{[x_{i-1},x_i]}$ is an affine map with $|p(x_k)-f(x_k)|=O(N^{-2})$ for $k=i-1,i$. For the reader's convenience, we provide a detailed proof.

Let $0\le t\le 1$ and choose $i$ such that $x_{i-1}\le t\le x_{i}$. Set $\widetilde{x}_i=x_i-N^{-2}$.
If $t>\widetilde{x}_{i}$, then $|t-x_i|<N^{-2}$ and it follows using \eqref{hdef}, \eqref{h'},  and $0\le f'(x^+)\le 1$, that
\begin{align*}
|p(t)-f(t)|&\le|p(t)-p(x_{i})|+|p(x_{i})-f(x_{i})|+|f(x_{i})-f(t)|\\
&\le 3\times O(N^{-2})=O(N^{-2}).
\end{align*}

It remains to consider the case $x_{i-1}\le t\le \widetilde{x}_i$. Write
\begin{equation}\label{ht}
\begin{split}
p(t)&=p(x_{i-1})+(t-x_{i-1})\frac{p(x_{i})-p(x_{i-1})}{x_{i}-x_{i-1}}\\
&=f(x_{i-1})+\bigl(p(x_{i-1})-f(x_{i-1})\bigr)+\\
&+(t-x_{i-1})\frac{f(\widetilde{x}_i)-f(x_{i-1})+\bigl(p(x_i)-f(\widetilde{x}_{i})+f(x_{i-1})-p(x_{i-1})\bigr)}{\widetilde{x}_i-x_{i-1}+N^{-2}}\,.
\end{split}
\end{equation}
Using \eqref{hdef}, the definition of $p$, $x_{i-1}\le t\le x_i$ and $0\le f'(x^+)\le 1$, we estimate
\begin{align*}
|p(x_{i-1})-f(x_{i-1})| &= O(N^{-2})\,,\\
|p(x_{i})-f(\widetilde{x}_i)| &= O(N^{-2})\,,\\
\left|\frac{t-x_{i-1}}{x_{i}-x_{i-1}}\right| &\le 1\,.
\end{align*}
We deduce from \eqref{ht} that
\begin{equation} \label{eq:eps1}
\left|p(t) - f(x_{i-1}) - (t-x_{i-1})\frac{f(\widetilde{x}_i)-f(x_{i-1})}{\widetilde{x}_i-x_{i-1}+N^{-2}} \right| = O(N^{-2}).
\end{equation}

Also, using $x_i-x_{i-1}\ge 2 N^{-2}$, $x_i\le t\le x_{i+1}$ and $0\le f'(x^+)\le 1$, we have
\begin{align}
\left|\frac{1}{\widetilde{x}_i-x_{i-1}}-\frac{1}{x_{i}-x_{i-1}}\right| &=O(N^{-2}|x_{i}-x_{i-1}|^{-2})\,,\label{eq:eps2}\\
\left|(t-x_{i-1})(f(\widetilde{x}_i)-f(x_{i-1}))\right| &= O(|x_i-x_{i-1}|^2)\,.\label{eq:eps3}
\end{align}
Combining the estimates \eqref{eq:eps1}--\eqref{eq:eps3}, we conclude that
\[
\left|p(t)-f(x_{i-1})-(t-x_{i-1})\frac{f(\widetilde{x}_{i})-f(x_{i-1})}{\widetilde{x}_i-x_{i-1}}\right|=O(N^{-2})\,.
\]
Since, on the other hand \eqref{1N} and \eqref{ab} yield that
\[\left|f(t)-f(x_{i-1})-(t-x_{i-1})\frac{f(\widetilde{x}_{i})-f(x_{i-1})}{\widetilde{x}_i-x_{i-1}}\right|=O(N^{-2})\,,\]
we have shown that \eqref{eq:finalclaim} holds.
\end{proof}

\begin{rem}
It follows from the previous proposition that if $N(\delta)$ is the minimum number of balls of radius $\delta$ needed to cover $\mathcal{C}^+$ in the Hausdorff metric, then $\log N(\delta) = O(\delta^{-1/2}|\log\delta|)$. This is close to being sharp: $\log N(\delta) =\Omega(\delta^{-1/2})$. Indeed, suppose again $N=\delta^{-1/2}$ is an integer. Set
\[
\Delta = \{  (a_1,\ldots, a_N)\, :\, a_i\in\{0,\tfrac1N,\ldots,1\},\, 0\le a_1\le \cdots \le a_N\le N \}.
\]
Then $\log\#\Delta=\Omega(N)$.
Given $a=(a_1,\ldots,a_N)\in\Delta$, let $L_a$ be the piecewise affine function satisfying $L_a(0)=0$ and $L_a'(t)=a_i/N$ for $(i-1)/N<t<i/N$. It is clear that $L_a\in\mathcal{C}^+$, and if $a\neq b\in\Delta$, then $D(L_a,L_b)= \Omega(1/N^2)$.

In particular, $(\mathcal{C}^+,D)$ has infinite box dimension, and is very far from being a doubling metric space (recall that a metric space is doubling if each ball can be covered by a uniformly bounded number of balls of half the radius).
\end{rem}

\subsection{Completion of the proof of Theorem \ref{thm:convex}} The proof of Theorem \ref{thm:convex} now follows the usual pattern, with minor variations. We still use the construction from Section \ref{sec:notation-and-construction}, but we now assume that $h(t)=t^\beta$ for some $\beta\in (5/3,2)$. As before, we ignore the elements of $\mathcal{C}^+$ whose graphs contain nontrivial dyadic line segments, but for simplicity of notation, we still denote the new slightly smaller collection by $\mathcal{C}^+$.

Lemma \ref{lem:large-deviation} holds for curves $\gamma\in\mathcal{C}^+$, as there is a uniform upper bound for the ratio $\mathcal{H}^1(\gamma\cap Q)/\diam(Q)$ (namely $2$) for all dyadic squares $Q$. The needed analog of Lemmas \ref{lem:many-lines-to-all-lines} and \ref{lem:many-algebraic-to-all-algebraic} is now the following.

\begin{lemma} \label{lem:many-convex-to-all-convex}
Let $1<\eta<4$. For each $n$, there is a  family of curves $\mathcal{C}_{n}\subset\mathcal{C}^+$, such that $\log\#\mathcal{C}_{n} = O(n 2^{\eta n/2})$ and, for any realization of $A$,
\begin{equation*}
\sup_{\gamma\in \mathcal{C}^+}Y^{\gamma}_n\le\sup_{\gamma\in \mathcal{C}_{n}}Y^{\gamma}_n +  O(2^{(3-\beta-\eta) n}).
\end{equation*}
\end{lemma}
\begin{proof}
Let $\delta=2^{-\eta n}$. We apply Proposition \ref{prop:dim-convex-functions} to obtain a $\delta$-dense family $\mathcal{C}'_n\subset\mathcal{C}^+$ such that $\log\#\mathcal{C}'_n= O(n 2^{\eta n/2})$.

We will modify $\mathcal{C}'_n$ in the same fashion as in the argument of Lemma \ref{lem:many-algebraic-to-all-algebraic}. If $f\in\mathcal{C}'_n$, let $(0,b]$ be the interval on which $f'(x^+)<2^{-n-1}$. If there is $k\in\N$ such that $|f(t)-k 2^{-n}|<2^{-\eta n+1}$ for some $0\le t\le b$, we choose for each $0\le i 2^{-n\eta}\le b$, $i\in\N\cup\{0\}$, numbers $y_i$ such that the function $f_i(x)=f(x)+y_i$ crosses the horizontal line $y=k 2^{-n}$ at $x_i=i 2^{-n\eta}$.
Let $\mathcal{E}_f$ be the collection of all $f_i$
and set
\[\mathcal{C}_n=\mathcal{C}'_n\cup\bigcup_{f\in\mathcal{C}'_n}\mathcal{E}_f\,.\]

Since each $\mathcal{E}_f$ has at most $O(2^{n\eta})$ elements, it follows that $\log\#\mathcal{C}_n= O(n 2^{\eta n/2})$.
Let $\gamma\in\mathcal{C}^+$. We claim that there is $\widetilde{\gamma}\in\mathcal{C}_n$ such that
\begin{equation}\label{Qsum}
\sum_{Q\in\mathcal{D}_n}\mathcal{H}^1(\gamma\cap Q)\le\mathcal{H}^1(\widetilde{\gamma}\cap Q)+O(2^{n(1-\eta)})\,.
\end{equation}
This implies the claim since then
\begin{align*}
Y_n^\gamma&=4^n P_n^{-1}\mathcal{H}^1(\gamma\cap A_n)\le 4^n P_n^{-1}\left(\mathcal{H}^1(\widetilde{\gamma}\cap A_n)+ O(2^{n(1-\eta)}\right)\\
&=Y_n^{\widetilde{\gamma}}+P_n^{-1}O(2^{n(3-\eta)})=Y_n^{\widetilde{\gamma}}+O(2^{n(3-\beta-\eta)})\,,
\end{align*}
recall that $P_n=\Theta(2^{n\beta})$.

To prove \eqref{Qsum}, fix $f\in\mathcal{C}^+$ and let $\gamma$ be the graph of $f$. We first choose $f^*\in\mathcal{C}'_n$ which is $\delta$-close to $f$. If $\mathcal{E}_{f^*}=\emptyset$, we let $\widetilde{f}=f^*$. Otherwise, there is $k\in\N$ and $0<t<b$ such that $|f^*(t)-k 2^{-n}|<2^{-\eta n+1}$, where $b=\sup\{x\,:\,f^*(x^+)<2^{-n-1}\}$. Denote
\[a=\sup\{x\,:\,f(x)\le k 2^{-n}\}\,,\]
with the convention $a=0$ if $f(0)>k 2^{-n}$.
It follows from the construction of $\mathcal{E}_{f^*}$ that we may choose $f^*_i\in\mathcal{E}_{f^*}$ and $0\le x_i=i 2^{-n\eta}\le b$ with $f^*_i(x_i)=k 2^{-n}$ such that
\begin{equation}\label{cases}
|x_i-a|\le\begin{cases}
O(2^{-n\eta})\text{ if }a\le b\\
O(2^{n(1-\eta)})\text{ if }a>b\,.
\end{cases}
\end{equation}
Let $\widetilde{f}=f^*_i$.

We first assume that $a>b$.
Let
$\gamma^1$, $\widetilde{\gamma}^1, \gamma^2$, and $\widetilde{\gamma}^2$ be the graphs of
$f_{[0,x_i]}$, $\widetilde{f}_{[0,x_i]}$, $f_{[a,1]}$, and $\widetilde{f}_{[a,1]}$, respectively. It then follows that
\begin{equation}\label{gest2}
\mathcal{H}^1(Q\cap\gamma^1)\le \mathcal{H}^1(Q\cap\widetilde{\gamma}^1)+O(2^{-\eta n})\text{ for all }Q\in\mathcal{D}_n.
\end{equation}
Recall that by Lemma \ref{lem:convex}, it is enough to bound  the Hausdorff distance of $\gamma^i\cap Q$ and $\widetilde{\gamma}^i\cap Q$ in order to estimate the difference $\mathcal{H}^1(\gamma^i\cap Q)-\mathcal{H}^1(\widetilde{\gamma}^i\cap Q)$.
Let $I\subset[a,1]$ be the set where the distance of $\widetilde{\gamma}$ is at least $2^{-n\eta}$ to all dyadic lines $y= j 2^{-n}$. Then
\begin{equation}\label{gest3}
\mathcal{H}^1(\gamma_{2}|_I\cap Q)\le\mathcal{H}^1(\widetilde{\gamma}_2\cap Q)+O(2^{-n\eta})\text{ for all }Q\in\mathcal{D}_n.
\end{equation}
Since the derivative of $\widetilde{f}$ is at least $2^{-n-1}$ on $[a,1]$ it follows that $|[a,1]\setminus I|=O(2^{n(1-\eta)})$. Combining with \eqref{gest2}, \eqref{gest3} and \eqref{cases}, and taking into account that each $\gamma$ intersects at most $O(2^n)$ squares in $\mathcal{D}_n$ yields \eqref{Qsum} in the case $a>b$.

If $a\le b$, we can repeat the above argument with $[0,x_i]$ and $[a,1]$ replaced by $[0,b]$ and $[b,1]$.
\end{proof}

\begin{proof}[Proof of Theorem \ref{thm:convex}]
We follow the proof of Theorem \ref{thm:main-technical-result}. Let
\begin{align*}
M_n &= \sup_{\gamma\in\mathcal{C}^+} Y_n^\gamma.
\end{align*}
Pick any $\eta\in (3-\beta,2(\beta-1))$; note the interval in question is nonempty thanks to our assumption that $\beta>5/3$. It will be enough to show that
\begin{equation} \label{eq:borel-cantelli-2}
\sum_{n=1}^\infty \PP\left(M_{n+1}-M_n>n^{-2}\sqrt{M_n}+ 2^{(3-\beta-\eta)n}\right)
<\infty.
\end{equation}
Indeed, thanks to Borel-Cantelli this implies that $\sup_{n\in\mathbb{N},\gamma\in\mathcal{C}^+} Y_n^\gamma<\infty$ almost surely, and from here the proof can be finished exactly as in the proof of Theorem \ref{thm:polynomial}.

From now on, fix $n$ and condition on $A_n$. Pick $\gamma\in\mathcal{C}^+$. Recall that under our assumptions, $P_n=\Theta(2^{n\beta})$. Lemma \ref{lem:large-deviation} (applied to curves in $\mathcal{C}^+$) yields that
\[
\PP\left(Y_{n+1}^\gamma - Y_n^\gamma > n^{-2}\sqrt{Y_n^\gamma}\right)  \le O(1) \exp(-\Omega(n^{-4}2^{(\beta-1)n})).
\]
Let $\mathcal{C}_n$ be the family given by Lemma \ref{lem:many-convex-to-all-convex}, with this $\eta$. Then
\begin{align*}
\PP\left(\max_{\gamma\in \mathcal{C}_n} Y_{n+1}^\gamma - M_n \ge  \frac{\sqrt{M_n}}{n^2} \right) &\le O(1) \exp(O(n 2^{\eta n/2})-\Omega(n^{-4}2^{(\beta-1)n}))\\
&=O(1)\exp(-\Omega(2^{n\eta'}))
\end{align*}
for any $0<\eta'<\beta-1$ (here we use that $\eta<2(\beta-1)$).
This implies \eqref{eq:borel-cantelli-2} and finishes the proof.
\end{proof}

\section{Generalizations}

We finish the paper by sketching some generalizations of the results in Section \ref{sec:introduction}.

In $\mathbb{R}^d$, $d>2$, Theorem \ref{thm:non-tube-null} can be generalized by considering tubes around planes rather than lines. For $k\in\{1,\ldots,d-1\}$, denote the Grassmannian of $k$-planes in $\mathbb{R}^d$ by $G(d,k)$. A {\em $G(d,k)$-tube} $T$ of width $w=w(T)$ is, as usual, a $w$-neighbourhood of a plane $V\in G(d,k)$.  We say that $A\subset\mathbb{R}^d$ is $G(d,k)$-tube null if for every $\delta>0$ one can find countably many $G(d,k)$-tubes $T_i$ covering $A$ with $\sum_i w(T_i)^{d-k}<\delta$. The proof of Theorem \ref{thm:gauge-main-result} extends to this setting to give:

\begin{theorem}\label{thm:plane-tube-null}
Let $h:(0,\infty)\to (0,\infty)$ be continuous and non-decreasing such that $h(2t)\le 2^d h(t)$, and
\[
\int_0^1 t^{-1}\sqrt{t^{k-d}|\log(t)| h(t)}\,dt<+\infty.
\]
Then almost surely the set $A$ constructed in Section \ref{sec:notation-and-construction} has no $G(d,k)$-tube null subsets of positive $\mathcal{H}^h$-measure.

In particular, there exist non $G(d,k)$-tube null sets of Hausdorff and box counting dimension $d-k$.
\end{theorem}
The latter claim is obtained by taking e.g.
\[
h(t)=t^{d-k}|\log t\log|\log t||^{-3}.
\]
Again, it is easy to see that the dimension threshold $d-k$ is sharp: any set $E\subset\mathbb{R}^d$ of dimension strictly less than $d-k$ is necessarily $G(d,k)$-tube null, since any particular orthogonal projection onto a $(d-k)$-dimensional subspace has zero $(d-k)$-dimensional Lebesgue measure.

As our main goal was to prove the existence of sets of small dimension that are not tube-null, we focused on a simple model that achieved this purpose. But it is possible to prove that many other sets arising from random models are not tube null (provided they are of sufficiently large dimension). This is true for a large class of repeated subdivision fractals; the key feature that must be present in the construction is that, conditioning on the $n$-th level, each surviving point has the same probability of surviving to the next level (and the partition elements should be regular enough that the combinatorial Lemma \ref{lem:many-lines-to-all-lines} can be carried through; but this is a mild condition). Thus, for example, classical fractal percolation limit sets with constant probabilities (see e.g. \cite{RamsSimon11}) are almost surely not-tube null when they have dimension strictly larger than $1$.

The main difference between the families $\mathcal{C}$ and $\mathcal{P}_k$ is their size: We have seen that the number of $\delta$-balls needed to cover $\mathcal{C}$ is $\exp(\Omega(\delta^{-1/2}))$ whereas for $\mathcal{P}_k$ only $\exp(O(|\log\delta|^2))$ such balls are needed. Theorem \ref{thm:polynomial} can be generalized to many other curve families satisfying such bounds. For instance, if $\mathcal{F}$ is a collection of curves in $\R^2$ such that for $0<\delta<1$ it can be covered by $\exp(O(|\log \delta|^{O(1)}))$ balls of radius $\delta$ (in the $D$ metric), and if each $\mathcal{F}$ is contained in a union of $O(1)$ curves in $\mathcal{C}$, then the proofs of Theorem \ref{thm:polynomial} and Lemma \ref{lem:many-convex-to-all-convex} can be combined to show the existence of non $\mathcal{F}$-tube null sets of dimension $1$.

Regarding higher dimensions, it seems likely that our methods can be used to prove results for algebraic curves and surfaces in $\R^d$ in the spirit of Theorem \ref{thm:plane-tube-null}.

We finish this discussion with a generalization in a different direction. Our proof of Theorem \ref{eq:growth-h} reveals that all orthogonal projections of the random measure $\mu$ onto lines are absolutely continuous, with a density bounded by some uniform random constant. It is natural to ask if the projections may enjoy any additional regularity. Away from the coordinate projections, one may use the method of Y. Peres and M. Rams in \cite{PeresRams11} to prove that projections have a H\"{o}lder continuous density (Peres and Rams prove this fact for projections of the natural measure on fractal percolation). However, the dyadic nature of the construction makes a discontinuity in the coordinate projections unavoidable. In a forthcoming work \cite{ShmerkinSuomala2012}, we address this issue by studying intersection properties of a different class of random measures, generated by removing a ``random soup'' consisting of countably many shapes generated by a Poisson point process, see e.g. \cite{NacuWerner11} for the description of this model. In particular, we show the existence of measures of dimension $1$ in $\mathbb{R}^2$, all of whose orthogonal projections onto lines are absolutely continuous, with a continuous density.

\bibliographystyle{plain}
\bibliography{tubenull}

\begin{thebibliography}{10}

\bibitem{ACP05}
Giovanni Alberti, Marianna Cs{\"o}rnyei, and David Preiss.
\newblock Structure of null sets in the plane and applications.
\newblock In {\em European {C}ongress of {M}athematics}, pages 3--22. Eur.
  Math. Soc., Z\"urich, 2005.

\bibitem{Carbery09}
Anthony Carbery.
\newblock Large sets with limited tube occupancy.
\newblock {\em J. Lond. Math. Soc. (2)}, 79(2):529--543, 2009.

\bibitem{CarberySoria97}
Anthony Carbery and Fernando Soria.
\newblock Pointwise {F}ourier inversion and localisation in {${\bf R}^n$}.
\newblock In {\em Proceedings of the conference dedicated to {P}rofessor
  {M}iguel de {G}uzm\'an ({E}l {E}scorial, 1996)}, volume~3, pages 847--858,
  1997.

\bibitem{CSV07}
Anthony Carbery, Fernando Soria, and Ana Vargas.
\newblock Localisation and weighted inequalities for spherical {F}ourier means.
\newblock {\em J. Anal. Math.}, 103:133--156, 2007.

\bibitem{Coste02}
Michel Coste.
\newblock An introduction to semialgebraic geometry.
\newblock Available at
  \texttt{http://perso.univ-rennes1.fr/michel.coste/polyens/SAG.pdf}, 2002.

\bibitem{Falconer03}
Kenneth Falconer.
\newblock {\em Fractal geometry}.
\newblock John Wiley \& Sons Inc., Hoboken, NJ, second edition, 2003.
\newblock Mathematical foundations and applications.

\bibitem{Furstenberg70}
Harry Furstenberg.
\newblock Intersections of {C}antor sets and transversality of semigroups.
\newblock In {\em Problems in analysis (Sympos. Salomon Bochner, Princeton
  Univ., Princeton, N.J., 1969)}, pages 41--59. Princeton Univ. Press,
  Princeton, N.J., 1970.

\bibitem{ManningSimon12}
Anthony Manning and K{\'a}roly Simon.
\newblock Dimension of slices through the {S}ierpinski carpet.
\newblock {\em Trans. Amer. Math. Soc.}, 365(1):213--250, 2013.

\bibitem{Mattila95}
Pertti Mattila.
\newblock {\em Geometry of sets and measures in {E}uclidean spaces}, volume~44
  of {\em Cambridge Studies in Advanced Mathematics}.
\newblock Cambridge University Press, Cambridge, 1995.
\newblock Fractals and rectifiability.

\bibitem{NacuWerner11}
{\c{S}}erban Nacu and Wendelin Werner.
\newblock Random soups, carpets and fractal dimensions.
\newblock {\em J. Lond. Math. Soc. (2)}, 83(3):789--809, 2011.

\bibitem{Orponen13}
Tuomas Orponen.
\newblock On the tube-occupancy of sets in $\mathbb{R}^d$.
\newblock {\em Int. Math. Res. Not. IMRN}.
\newblock to appear. Preprint, available at
  \texttt{http://arxiv.org/abs/1311.7340}.

\bibitem{PeresRams11}
Yuval Peres and Micha{\l} Rams.
\newblock Projections of the natural measure for percolation fractals.
\newblock Preprint, available at \texttt{http://arxiv.org/abs/1406.3736}, 2014.

\bibitem{PeresSolomyak05}
Yuval Peres and Boris Solomyak.
\newblock The sharp {H}ausdorff measure condition for length of projections.
\newblock {\em Proc. Amer. Math. Soc.}, 133(11):3371--3379 (electronic), 2005.

\bibitem{RamsSimon11}
Micha{\l} Rams and K\'{a}roly Simon.
\newblock Projections of fractal percolations.
\newblock {\em Ergodic Theory Dynam. Systems}.
\newblock to appear. DOI 10.1017/etds.2013.45.

\bibitem{RamsSimon14}
Micha{\l} Rams and K{\'a}roly Simon.
\newblock The {D}imension of {P}rojections of {F}ractal {P}ercolations.
\newblock {\em J. Stat. Phys.}, 154(3):633--655, 2014.

\bibitem{ShmerkinSuomala2012}
Pablo Shmerkin and Ville Suomala.
\newblock Spatially independent martingales, intersections, and applications.
\newblock Preprint, available at \texttt{http://arxiv.org/abs/1409.6707}, 2014.

\end{thebibliography}

\end{document}